\documentclass{amsart}

\usepackage{amsfonts}
\usepackage{amscd}
\usepackage{amssymb}

\usepackage[hidelinks]{hyperref}

\newtheorem{theorem}{Theorem}[section]

\newtheorem{corollary}[theorem]{Corollary}
\newtheorem{proposition}[theorem]{Proposition}
\theoremstyle{remark}

\newtheorem{remark}[theorem]{Remark}
\numberwithin{equation}{section}

\begin{document}

\author{Konstantinos A. Draziotis}
\title[Practical solution of some families of quartic diophantine]{Practical solution of some families of quartic and sextic diophantine hyperelliptic equations}
\address {K. A. Draziotis \\
Department of Informatics\\
Aristotle University of Thessaloniki \\
54124 Thessaloniki, Greece}
\email{drazioti@csd.auth.gr}


\maketitle

\begin{abstract}
Using elementary number theory we study Diophantine equations over the rational integers of the following form,
$y^2=(x+a)(x+a+k)(x+b)(x+b+k)$, $y^2=c^2x^4+ax^2+b$ and $y^2=(x^2-1)(x^2-\alpha^2)(x^2-(\alpha+1)^2).$   We express their integer solutions by means of the divisors of the discriminant of $f(x),$ where $y^2=f(x)$.
\end{abstract}

\section{Introduction}
Let $f(x)\in \Bbb{Z}[x]$ be a monic polynomial which is not a square. We consider the hyperelliptic curve, 
\begin{equation}
y^{2}=f(x), \ 2|\deg f\ {\rm{and}}\ \deg{f}\ge 4. \label{eq}
\end{equation}
One way to study the integer points of curve (\ref{eq}) is to use the so called Runge's method \cite{Ayad,GH,Runge,sprind} (see Appendix A for a brief summary of this method). In fact, Runge proved the finiteness of the integer points of equation (\ref{eq}) in 1887 \cite{Runge}\footnote{A  more general result was proved, where our curve is a special case. See Appendix C for the general curves that satisfy Runge's condition.}.  

In \cite{HiSt,Walsh,Le}, using Runge's method with a combination of an effective version of Eisenstein theorem  \cite{Sch}, the authors provided  a uniform upper bound for the size of their integral points on curves of the form  $(\ref{eq}).$ 

In \cite{Tengely}, the author considers more general equations of the form
$$F(x)=G(y),\ F,G\ {\rm monic}, \ F(x)-G(y)\ {\rm{irreducible \ }} \text{in}\ {\mathbb{Q}}[x,y],$$
where $\deg{F}=n,\ \deg{G}=m$ $(m\ge n)$ are such that, $\gcd(n,m)>1.$ With $H(F)$ we denote the height of $F$ i.e. the maximum of the absolute values of the coefficients of polynomial $F$. Then (using big-O notation), 
$$\max(|x|,|y|) = O\big((2hm)^{4m^2}\big)\ \ (h=\max(H(F),H(G)),$$ 
for some effective computable constant. For the full details of the bound see  \cite[Theorem 2.2.1]{Tengely} and for detailed references of Runge's method see, \cite[Section 2.1]{Tengely}.

For the special case $\deg{f}=4,$ we get the quartic hyperelliptic curve,
\begin{equation}\label{deg4}
y^2=x^4+ax^3+bx^2+cx+d.
\end{equation}
Masser showed that (see, \cite{Masser}), 
$$|x|\le 26H(f)^3,$$
where $H(f)=\max\{|a|,|b|,|c|,|d|\}.$  In fact he proved something stronger. If $$X(H) = \max\{|x|\in {\mathbb{Z}} : \text{there\ is\ an\  integer\ }y\ \text{such\ that, } y^2=f(x)\ \text{and\ } H(f)\le H\},$$
then 
$$H^3/4096\le X(H)\le 26H^3,$$
i.e. $X(H)=\Theta(H^3).$
 We note that, the method suggested by Masser does not use Runge's method. A practical method based in Masser's method was provided in \cite{poulakis} and generalized in \cite{laszlo}. 
 
 The case where $f(x)$ is not a monic polynomial, i.e. the case of a general quartic $y^2=f(x),$ where $f(x)=Ax^4+Bx^3+Cx^2+Dx+C,$ can be treated using linear forms of elliptic logarithms, for instance see \cite{tzanakis}. Finally, if $f(x)$ is reducible, then we can use descend methods (\cite{duq}) combining with Chevalley-Weil theorem (\cite{Chev,draz2}) in order to solve them. If $f(x)$ is of degree 3 and reducible then again descend methods can be applied (see \cite{draz3}).
 \\\\
{\bf Our Contribution}. In the present work we provide explicit formulas for the integer solutions to equations of the following form: $y^2=(x-\alpha)(x-\beta)(x-\gamma)(x-\delta),$ where $\beta-\alpha=\delta-\gamma=cons.\not=0$ and $y^2=c^2x^4+ax^2+b,$  $a,b,c\in {\mathbb{Z}}.$ Note that, the previous  families belong to the curves that can be treated with Runge's method (see Appendix C), however the general upper bounds are in practice infeasible. The first family  belongs to the class of equations where the methods presented in \cite{Masser,poulakis} can treat. However, we shall show that, this is not (always) practical. Also, the second family cannot be treated with Masser's method.  Furthermore, we show that using our method, we can easily solve the sextic Diophantine equation $y^2=(x^2-1)(x^2-a^2)(x^2-(a+1)^2), \ a\in{\mathbb{Z}}.$
 \\\\
{\bf Roadmap}. In the next section we provide our results in two Propositions. The following equations are studied:\\
$({\rm i}).$ $y^2=(x+a)(x+a+k)(x+b)(x+b+k),\ a,b,k\in{\mathbb{Z}},$\\
$({\rm ii}).$ $y^2=c^2x^4+ax^2+b,\ a,b,c\in{\mathbb{Z}},$\\
We provide the proofs, some examples, and an implementation in sagemath \cite{sage}. What is more, we provide an example of a sextic hyperellipitc curve\\
$({\rm iii}).$  $y^2=(x^2-1)(x^2-a^2)(x^2-(a+1)^2), \ a\in{\mathbb{Z}}.$

Section \ref{sec:conclusion}, is the last section in which we provide some concluding remarks. Also, there are three appendices, where we provide a brief summary of Runge's method, a reminder of Massers's method \cite{Masser}, and the definition of {\it{Runge type}} curve.

\section{Our results} 
\subsection{Main idea}\label{main_idea} Let $y^2=f(x),$ with $f(x)\in {\mathbb{Z}}[x]$  not a square, $\deg{f}=4,$ and $D(z)={\rm Res}_{x}(f(x)+z,f'(x)),$ where with ${\rm Res}_x(\cdot,\cdot)$ we denote the resultant with respect to indeterminate $x.$ The polynomial $D(z)\in {\mathbb{Z}}[z],$ has degree $3$ and assume that there exists $z_0\in {\mathbb{Q}},$ which is a double root of the equation $D(z)=0.$ Then,  under some plausible geometric conditions, $f(x)+z_0=r(x)^2,$ for some $r(x)\in {\mathbb{Q}}[x].$ In order this last identity to be useful to us,  we need one further assumption, except the existence of a rational double root of $D(z)=0.$ We assume  that $z_0=A/B^2,$ with $\gcd(A,B)=1.$ Let $(x_0,y_0)\in {\mathbb{Z}}^2$ such that, $y_0^2-f(x_0)=0.$ 

On the other hand, we have $B^2f(x)+A=r(x)^2,$ but now $r(x)\in {\mathbb{Z}}[x].$ Thus,
$$B^2(y_0^2-f(x_0))=0\Leftrightarrow B^2y_0^2-(r(x_0)^2-A)=0 \Leftrightarrow (r(x_0)-By_0)(r(x_0)+By_0)=A.$$
Then, using factorization we can calculate explicit formulas for $x_0$ and $y_0.$

In order to shed more light on the {\it{geometric condition}}, we provide two examples. Let $f(x)=49x^4-15x^2-2$, then 
$$D(z)={\rm Res}_{x}(f(x)+z,f'(x))=38416(z - 2)(196z - 617)^2,$$
and $f(x)+z_0=f(x)+\frac{617}{196}=	\frac{1}{196} (98x^2 - 15)^2.$

If $g(x)=x^4 + 225x^3 + 49,$ then 
$$D(z)={\rm Res}_{x}(g(x)+z,g'(x))=(256z - 69198034331)(z + 49)^2,$$
and  $g(x)+z_0=g(x)-49=x^3(x+225).$
I.e. both $f(x), g(x)$ have a square in the decomposition of their resultant $D(z)$, but $g(x)+z_0$ does not contain a square in its decomposition. This is because the point $(0,-49)$ of the curve $z+g(x)=0$ is inflexion point.

\subsection{The curve $y^2=(x+a)(x+a+k)(x+b)(x+b+k)$}
We begin with the study of the equation:
\begin{equation}\label{eq1}
y^2=f(x)=(x+a)(x+a+k)(x+b)(x+b+k).
\end{equation}
We shall prove the following Proposition.
\begin{proposition}
Let $a,b,k,$ be three integers with $a\not=b$ (if $a=b$, then $f(x)$ is a square). We set 
$$C=2ab+a+b-(d_1+d_2)/2\ \text{and}\ \Delta = 4\big((a+b+1)^2-2C\big),$$ where $d_1,d_2$ are integers of the same parity and such that $d_1d_2=(ka-kb)^2.$  If $\Delta=\delta^2,$ for some $\delta\in {\mathbb{Z}},$ then, the integer solutions of Diophantine equation
$y^2=f(x),$  are of the form:
$$(x_0,|y_0|)=\bigg(-\frac{a+b+k}{2}\pm \frac{\delta}{4},\bigg{|} \frac{d_1-d_2}{4}\bigg{|}\bigg).$$ 
If $\Delta$ is not a square, the equation does not have any integer solution for the specific $d_1,d_2$.
\end{proposition}
\begin{proof}
Let $G(x,y)=y^2-f(x)$ and\footnote{For a symbolic computation of the resultant you can see:\\ \url{https://github.com/drazioti/simple_quartic/tree/main/resultant}} 
$$D(z)={\rm Res}_x(f(x)+z,f'(x))=M(z)\Big(\big(k(a-b)\big)^2-4z\Big)^2,$$
for some linear polynomial $M(z)$. We set $z_0=\frac{(k(a-b))^2}{4},$ then
$$4(G(x,y)-z_0)=4G(x,y) - (k(a-b))^2=-(r(x)-2y)(r(x)+2y),$$ 
where 
$$r(x)=2x^2+2x(a+b+k)+2ab+ka+kb.$$
Let $(x_0,y_0)\in{\mathbb{Z}}^2$ such that, $G(x_0,y_0)=0.$ Then, there exist integers $d_1$ and $d_2$ such that,
$$r(x_0)-2y_0=d_1, r(x_0)+2y_0=d_2, d_1d_2=(k(a-b))^2,$$
and so,
$$r(x_0)=(d_1+d_2)/2, y_0=(d_2-d_1)/4, d_1d_2=(k(a-b))^2.$$
We set
$$B=2(a+b+k)\ \text{and}\ $$
from the hypothesis,
$$C=2ab+ka+kb-(d_1+d_2)/2,$$
thus we get, 
$$2x_0^2+Bx_0+C=0, y_0=(d_2-d_1)/4, d_1d_2=(k(a-b))^2.$$
From the second equality we see that $d_1,d_2$ must have the same parity.
Since $\Delta=\delta^2,$ we get  
$$x_0=\frac{-2(a+b+k)\pm \delta}{4}=\frac{-(a+b+k)}{2}\pm \frac{\delta}{4}.$$
The result follows.
\end{proof}
\begin{remark}
Note that, $k^2(a-b)^2$ divides the discriminant of $f(x).$ 
\end{remark}

Now, we easily get the following pseudocode.\\\\
{\footnotesize{
\noindent
{\bf Input}:  $a,b,k$\ $(a\not=b)$
\ \\
{\bf Output}: The integer solutions of $y^2=(x+a)(x+a+k)(x+b)(x+b+k),$ with $y\ge 0.$ \\\\
	\texttt{01.}  \texttt{DIV$\leftarrow \{n\in {\mathbb{Z}}: n|(ka-kb)^2 \}$}\\
	\texttt{02.}  \texttt{$B\leftarrow 2(a+b+1)$}\\
	\texttt{03.}    $L=[\ ]$ \# this is the list where we keep the solutions\\
    \texttt{04.}   {\texttt{For $d_1$ in DIV}}\\
	\texttt{05.}    \hspace{0.7cm} $d_2\leftarrow (ka-kb)^2/d_1$\\
	\texttt{06.} \hspace{0.7cm}  \texttt{$C\leftarrow 2ab+a+b-(d_1+d_2)/2$}\\
	\texttt{07.} \hspace{0.7cm}  \texttt{$\Delta\leftarrow B^2-8C$}\\
	\texttt{08.}    \hspace{0.7cm} If ($d_1\equiv d_2\bmod{2}$) AND ($\Delta$ is a square, say $\delta^2$) AND ($d_1\le d_2$)\\
	\texttt{09.}    \hspace{1.4cm} $x_{1,2}\leftarrow -\frac{a+b+k}{2}\pm \frac{\delta}{4}$ \\
	\texttt{10.}    \hspace{1.4cm} $y\leftarrow \frac{d_2-d_1}{4}$\\
	\texttt{11.}    \hspace{1.4cm} Append list $L$ with $(x_1,y), (x_2,y)$\\
	\texttt{12.}    return $L$
}}

It is easy to implement the previous algorithm in order to find the integer solutions\footnote{For instance, see \url{https://github.com/drazioti/simple_quartic/blob/main/1.py}, for an implementation in sagemath \cite{sage}.}. 

All the steps of the previous algorithm can be easily computed, except maybe calculation of the set ${\rm{DIV}}$ in line 01. Thus, the complexity is dominated by the the complexity of the algorithm which computes divisors of $(ka-kb)^2.$ I.e. we need the factorization of $|k(a-b)|$ which for large $k,a-b$ can be computationally intensive.

We provide two examples. If $a=1, b=2, k=41,$ we have the equation,
$$y^2=x^4 + 88x^3 + 2063x^2 + 5588x + 3612$$
and the integer solutions are (for $y\ge 0)$:
$$\{(7, 420), (-51, 420), (-22, 420), (-1, 0), (-43, 0), (-2, 0), (-42, 0)\}.$$
For the equation, 
$$y^2=x^4 + 20x^3 + 97x^2 - 30x - 504,\ (a,b,k)=(3,-2,9)$$
we get (for $y\ge 0$), 
$$\{(9, 168), (-19, 168), (3, 30), (-13, 30), (-5, 14), (-4, 12), (-6,12),$$
$$(2, 0), (-12, 0), (-3, 0), (-7, 0)\}.$$
\begin{remark}
If $f(x)=(x+a)(x+a+k)(x+b)(x+b-k),$ then the study of $y^2=f(x)$ can be treated in a similar way. 
\end{remark}
\begin{remark} 
If we consider $a=0,\ b=2^{\ell}$ and $k=1,$ and we use the method of  \cite{poulakis}, we get an interval that contains $[0,2^{\ell}].$ In this interval we have to search for $x,$ checking one by one if it provides an integer solution $(x,y).$ So, this method is not practical for large $\ell.$ On the other hand, the number of divisors of $(ka- kb)^2=2^{2\ell}$ is $2(2\ell+1).$ So the complexity is $O(\ell),$ whereas the complexity of \cite{poulakis} is $O(2^{\ell}).$ 

In general, from the previous Proposition we easily get that the number of integer solutions of the Diophantine equation (\ref{eq1}) is at most $2\tau\big(k^2(a-b)^2\big),$ where $\tau(w)$ is the number of positive divisors
of integer $w.$
\end{remark}
\begin{corollary}
If $(x,y)$ is an integer solution of $y^2=(x+a)(x+a+k)(x+b)(x+b+k),$ and $M=\max\{|a|,|b|,|k|\}$, then 
$|x|<4M^2.$
\end{corollary}
\begin{proof}
Since the integer solutions are of the form
$x=-\frac{a+b+k}{2}\pm \frac{\delta}{4}$ we have to bound 
the quantities $|a|,|b|,|k|,|\delta|.$
From the definition of $M$ we get $|a|,|b|,|k|\le M.$ Now, $\delta^2=B^2-8C,$ where
$|B|=2|a+b+1|<6M$ and $|C|=|2ab+a+b-(d_1+d_2)/2|$ for $d_i|(k(a-b))^2.$ So, 
$$|C|<2M^2+2M+\max\{|d_1|,|d_2|\},\ \text{thus}\ |C|<2M^2+2M+4M^4<8M^4.$$
We conclude that 
$$\delta^2 = B^2-8C<B^2+8|C|<36M^2+64M^4<100M^4.$$
Finally,
$$|x|<\big|\frac{a+b+k}{2}\big|+\big|\frac{\delta}{4}\big|<3M/2+10M^2/4<4M^2.$$
\end{proof}
\subsection{The curve  $y^2=c^2x^4+ax^2+b$}
We continue our study with the equation $y^2=f(x)=c^2x^4+ax^2+b,$ where $a,b,c,$ are integers and $f(x)$ is not a square.
\begin{proposition}\label{Prop:2}
Let $a,b,c,$ be three integers with $c\not=0, \delta=a^2-4bc^2\not=0.$ We set $f(x) = c^2x^4+ax^2+b$ and $G(x,y)=y^2-f(x).$ Then, any integer solution $(x_0,y_0)$ of the Diophantine equation $G(x,y)=0$, is of the form:
$$(x_0,|y_0|)=\bigg(\pm \sqrt{\frac{d_1+d_2-2a}{4c^2}},\bigg{|} \frac{d_1-d_2}{4c}\bigg{|}\bigg),$$
for some $d_1,d_2,$ integers of the same parity such that, $d_1d_2=\delta$ and assuming that the square root exists.
\end{proposition}
\begin{proof}
We have $$D(z)={\rm Res}_x(f(x)+z,f'(x))=16c^4(z+b)(4c^2z-\delta)^2,$$
and we set  $z_0$ be the double root $\frac{\delta}{4c^2}.$ Then,
$$4c^2(G(x,y)-z_0)=4c^2G - \delta=-(r(x)-2cy)(r(x)+2cy),$$ where 
$r(x)=2c^2x^2+a.$ Without loss of generality, with $c$ we write the positive square root of $c^2.$

Let $(x_0,y_0)\in{\mathbb{Z}}^2$ such that, $G(x_0,y_0)=0.$ Then, 
$$(r(x_0)-2cy_0)(r(x_0)+2cy_0)=\delta.$$
Thus, there exist two integers $d_1,d_2,$ such that,
$$r(x_0)-2cy_0=d_1, r(x_0)+2cy_0=d_2, d_1d_2=\delta.$$
Therefore, 
$$r(x_0)=\frac{d_1+d_2}{2}, y_0=\frac{d_2-d_1}{4c}, d_1d_2=\delta.$$
From the second equality we note that, $d_1,d_2$ must have the same parity, else $y_0\not\in {\mathbb{Z}}.$ 
We conclude therefore,
$$x_0=\pm \sqrt{\frac{d_1+d_2-2a}{4c^2}}.$$
If $x_0$ is integer, we end up with an integer solution of $y^2=f(x).$
\end{proof} 
We provide the pseudocode.\\\\
{\footnotesize{
\noindent
{\bf Input}\footnote{For instance see \url{https://github.com/drazioti/simple_quartic/blob/main/2.py}}:  $c,a,b$
\ \\
{\bf Output}: The integer solutions of $y^2=f(x)=c^2x^4+ax^2+b,$ with $y\ge 0.$ \\\\
	\texttt{01.}  \texttt{$\delta \leftarrow a^2-4bc^2$}\\
	\texttt{02.}  \texttt{DIV$\leftarrow \{n\in {\mathbb{Z}}: n|\delta \}$}\\
	\texttt{03.}    $L=[\ ]$\\
    \texttt{04.}   {\texttt{For $d_1$ in DIV}}\\
	\texttt{05.}    \hspace{0.7cm} $d_2\leftarrow \delta/d_1$\\
	\texttt{06.} \hspace{0.7cm} If ($d_1\equiv d_2\bmod{2}$)\\
	\texttt{07.}    \hspace{1.4cm} $K\leftarrow \frac{1}{2c^2}\Big(\frac{d_1+d_2}{2}-a\Big)$\\
	\texttt{08.}   \hspace{1.4cm} If ($K$ is a square integer, say $m_{\delta}^2$) AND ($d_1\le d_2$)\\
	\texttt{09.}    \hspace{2.1cm} $x_{1,2}\leftarrow \pm m_{\delta}$ \\
	\texttt{10.}    \hspace{2.1cm} $y\leftarrow \frac{d_2-d_1}{4|c|}$\\
	\texttt{11.}    \hspace{2.1cm} Append list $L$ with $(x_1,y), (x_2,y)$\\
	\texttt{12.}    return $L$
}}

\begin{remark}
In magma \cite{Magma}, the command \texttt{SIntegralLjunggrenPoints([D,A,B,C],
[])}, provides the integral points 
on the curve $C : Dy^2 = Ax^4 + Bx^2 + C,$  provided that $C$ is nonsingular. Furthermore, \texttt{IntegralQuarticPoints([a,b,c,d,e])} provides the integral points 
on the curve $C : y^2 = ax^4 + bx^3 +cx^2+dx+e.$
\end{remark}
\begin{remark}
For $c=1$ someone can apply the method of \cite{poulakis}.  For instance, if we consider the equation $y^2=x^4-2^{\ell}x^2+1,$ then following the method of \cite{poulakis} we have to search the integer solutions $x$ in the interval $[-2^{\ell},2^{\ell}],$ which is exponentially large. In our case we have to find the divisors of the integer $2^{2\ell}-4,$ which is feasible for all $\ell,$ say $80\le \ell \le 120.$ So our method is practical, whereas method \cite{poulakis} is infeasible for this case\footnote{In sagemath, with the following code we can compute extremely fast, the prime factorization of $2^{2\ell}-4,$ for all $\ell \in [80,120] :$ \texttt{[[k,factor(2**(2*k)-4)] for k in range(80,121)]}. \\ In \url{https://github.com/drazioti/simple_quartic/blob/main/4.txt} we computed all the integer points of the curves $y^2=x^4-2^{\ell}x^2+1$  for $80\le \ell \leq 120.$}. Furthermore, the case $c=1$ can be treated by the original Masser's method with factorization, see Appendix B. 
\end{remark}
\begin{remark}
In subsection \ref{main_idea} we demanded that the double root of the resultant is of the form $A/B^2,$ with $\gcd(A,B)=1.$ In the proof of the previous Proposition, we prove that there is always a double root, $z_0=\frac{a^2-4bc^2}{4c^2}.$ But, here it may occur $\gcd(a^2-4bc^2,4c^2)>1$. We can see that always $\gcd(a^2-4bc^2,4c^2)$ is a square, so after we delete the gcd from the numerator and denominator, we get $z_0=A'/B'^2,$ with $\gcd(A',B')=1.$ 
\end{remark}
Now, we can easily study the Diophantinte equation
$cy^2=cx^4+ax^2+b.$  
\begin{corollary}
Let $a,b,c,$ be three integers with $ac\not=0$ and $\Delta=a^2-4bc\not=0.$ We set $h(x) = cx^4+ax^2+b.$ Then, the integer solutions of the Diophantine equation $cy^2=h(x)$, are of the form:
$$(x_0,|y_0|)=\bigg(\pm\sqrt{\frac{d_1+d_2-2a}{4c}},\bigg{|} \frac{d_1-d_2}{4c}\bigg{|}\bigg),$$
for some $d_1,d_2$ integers of the same parity such that $d_1d_2=\Delta$ and assuming that the square root exists.
\end{corollary}
\begin{proof}
It's the same as in Proposition \ref{Prop:2}, by making the substitution $c^2\rightarrow c.$
\end{proof}
Similarly, as previous we get the following pseudocode.\\
{\footnotesize{
\noindent
{\bf Input}\footnote{For an implementation in sagemath see,\\ \url{https://github.com/drazioti/simple_quartic/blob/main/3.py}.\\ Also, a sagemath implementation that treats all the cases considered here and the general case $y^2=x^4+ax^3+bx^2+cx+d$ is provided in the link:\\ \url{https://github.com/drazioti/simple_quartic/blob/main/general_quartic.py}}:  $c,a,b$
\ \\
{\bf Output}: The integer solutions of $cy^2=cx^4+ax^2+b,$ with $y\ge 0.$ \\\\
	\texttt{01.}  \texttt{$\Delta \leftarrow a^2-4bc$}\\
	\texttt{01.}  \texttt{DIV$\leftarrow \{n\in {\mathbb{Z}}: n|\Delta \}$}\\
	\texttt{02.}    $L=[\ ]$\\
    \texttt{03.}   {\texttt{For $d_1$ in DIV}}\\
	\texttt{04.}    \hspace{0.7cm} $d_2\leftarrow \Delta/d_1$\\
	   \texttt{05.} \hspace{0.7cm} If ($d_1\equiv d_2\bmod{2}$) AND ($d_1\le d_2$)\\
	\texttt{06.}    \hspace{1.4cm} $K\leftarrow \frac{d_1+d_2-2a}{4c}$\\
	\texttt{07.}   \hspace{1.4cm} If ($K$ is a square integer, say $m_{\delta}^2$)\\
	\texttt{08.}    \hspace{2.1cm} $x_{1,2}\leftarrow \pm m_{\delta}$ \\
	\texttt{09.}    \hspace{2.1cm} $y\leftarrow \frac{d_2-d_1}{4c}$\\
	\texttt{10.}    \hspace{2.1cm} Append list $L$ with $(x_1,y), (x_2,y)$\\
	\texttt{11.}    return $L$
}}

For instance, for $c=6, a=13, b=2$ we get the curve\footnote{Also, you can try the following code in magma :\\ 
\texttt{C :=6;
A :=13;
B :=2;
SIntegralLjunggrenPoints([C,C,A,B],[])}. \\
You may use the online calculator \url{http://magma.maths.usyd.edu.au/calc/}} $6y^2=6x^4+13x^2+2.$ We compute its integer points $(y\ge 0$) : $(\pm 2,5).$ For $c=12, a=-30, b=-24$ we get $(\pm 2,2).$ 
\subsection{One example of sextic hyperelliptic curve}
The following example concerns a sextic hyperelliptic equation, where our method can easily be applied. Let the curve $C:y^2=f(x),\ f(x)=(x^2-1)(x^2-4)(x^2-9).$
We remark that 
$${\rm{Res}}_x(f(x)+Z,f'(x))=64(27Z + 400)^2(Z - 36)^3.$$ 
Although the denominator of $z_0$ (the double root) is not square we can work with the other factor, i.e. $(Z-36)^3.$ Then,
$F=y^2-f(x),$ is such that 
$$F(x,y)-36=-(x^3 - 7x + y)(x^3 - 7x - y).$$
Let $(a,b)\in C({\mathbb{Z}}),$ then 
$$(a^3 - 7a + b)(a^3 - 7a - b)=36.$$
We get 
$$a^3 - 7a + b = d_1,\ a^3 - 7a - b = d_2$$
where $(|d_1|,|d_2|)=(1,36),(36,1),(2,18),(18,2), (3,12),(12,3),(4,9),(9,4),(6,6).$
So, 
$$2(a^3-7a)= d_1+d_2\in \{\pm 37,\pm 20, \pm 15, \pm 14, \pm 12\},$$ therefore
$$a^3-7a=\pm 10, \pm 7, \pm 6.$$ Only the equation $a^3-7a=\pm 6$ has integer solutions, $a=\pm 1, \pm 2, \pm 3.$ So, we get only the trivial solutions.

This example suggests that the equation 
$$y^2=(x^2-1)(x^2-a^2)(x^2-(a+1)^2)=M(x),$$
can be treated with a similar way. Indeed, if $G(x,y)=y^2-M(x),$ then the triple root of the resultant is,
$$z_0 = a^4 + 2a^3 + a^2 $$ and
$$G(x,y)-z_0 = -(-x^3 + xa^2 + xa + x - y)(-x^3 + xa^2 + xa + x + y).$$
We continue as in the example.

\begin{remark}
In \cite{laszlo} the author provides an algorithm for finding the integer points in the more general Diophantine equation (\ref{eq}) i.e. of the form,
$$y^2=x^{2k}+a_{2k-1}x^{2k-1}+\cdots +a_1x+a_0,$$
where $k$ is a positive integer. This method again has its roots in the paper of Masser \cite{Masser} and in \cite{poulakis}. Therefore, it is not based on Runge's method. A further generalization was given in \cite{laszlo2}.

\end{remark}

\section{Conclusion}\label{sec:conclusion}
In the present work we provided explicit formulas for the integer solutions of quartic hyperelliptic curves of the form $y^2=(x+a)(x+a+k)(x+b)(x+b+k)$ and $y^2=c^2x^4+ax^2+b.$ The formulas depend on the divisors of the discriminant of $f(x),$ where $y^2=f(x)$ is our curve. We used suitable factorization of the previous equations, and elementary number theory to find their integer solutions. The other methods used to practically solve the previous Diophantine equations, apply brute force in a suitable interval. Furthermore, we studied the sextic $y^2=(x^2-1)(x^2-\alpha^2)(x^2-\beta^2),$ with $\beta-\alpha=1.$

\ \\
{\LARGE{{\sc Appendix}}}\\
\setcounter{equation}{0}

{\bf{A. Runge's Method}}\\
Runge's method uses Puiseux series to study some classes of Diophantine equations (equations of the form $(\ref{eq})$ are of this type). We carry out the following steps :\\
{\it 1.} Using Puiseux expansion theorem we can find a polynomial $g(x)$ and a power series $S(T),$
such that 
$$y-g(x)=S\bigg(\frac{1}{x}\bigg),\  \text{where}\ S(T)\in {\mathbb{Q}[[T]]}$$ and $g(x)\in{\mathbb{Q}}[x].$\\
{\it 2.} From the form of $S(T)$ we can find a positive constant $A=A(F)$ such that,
if $|x|>A(F),$ then $|S(\frac{1}{x})|<c|x|^{-\rho},$ for some positive integer $\rho$ and a positive real number $c.$\\
{\it 3.} If $(a,b)\in C({\mathbb{Z}}),$ then for $|a|>\max\{A(F),\sqrt[\rho]{c}\},$ we get $|b - g(a)|<1.$\\
{\it 4.} Now, either we apply an effective version of Eisenstein's theorem \cite{Sch} and so we shall get a uniform bound for $|a|$ or with some add hoc method we explicit calculate the denominators of the coefficients of $g(x).$ Thus, multiplying say by $w,$ we conclude with the inequality
$|wb-wg(x)|<w.$ We set $\mu(x)=wg(x)\in {\mathbb{Z}}[x].$
Then, we solve the finitely many equations
$(wb)^2 = (\mu(x)+r)^2,$ with $r=-w+1,...,w-1.$
Since $b^2=f(a),$ we end up with the equations (with one unknown)
$$w^2f(a)-(\mu(a)+r)^2 = 0,\ r=-w+1,\dots,w-1.$$
Thus, if the constant $A(F)$ is {\it{small}}, we can have a practical algorithm for finding the set $C({\mathbb{Z}}).$ In \cite{draz,osipov,osipov_russian} the previous method allows to practically solve some Diophantine equations e.g. $x^{nr}+y^{n}=q$  in \cite{draz} or $x^4 - x^2y - xy^2 - y^2 + 1 = 0$ in \cite{osipov}.  	
So this method, except the uniform bounds that provides, sometimes it may also be appropriate in order to get a practical algorithm for the integer points.\\

{\bf{B. Masser's Method for $y^2=x^4+bx^2+d.$}}\\
We follow \cite{Masser}. Let $y^2=f(x),$ where $f(x)=x^4+ax^3+bx^2+cx+d\in {\mathbb{Z}}[x].$
Put, 
$$e=4b-a^2,\ C=64c-8ae,\ D=64d-e^2,\ \text{and\ } Q(x)=8x^2+4ax+e.$$ Then, the following identity holds,
\[ 64f(x) - (Q(x))^2 = Cx + D.\]
In our case, $a=c=0,$ so $e=4b, C=0,\ D=64d-16b^2$ and $Q(x)=4(2x^2+b).$ Thus, the previous identity is written,
$$4f(x)-(2x^2+b)^2=4d-b^2.$$
If $(x_0,y_0)$ is an integer point, then 
$$4f(x_0)-(2x_0+b)^2=4d-b^2,\ \text{so}\ (2y_0)^2-(2x_0+b)^2=4d-b^2.$$ Now, using elementary number theory we can find $(x_0,y_0).$\\

{\bf{C. Runge's Condition.}}\label{appendixC}\\
Let $F\in {\mathbb{Z}}[x,y]$ be an irreducible polynomial and 
$$F(x,y)=\sum_{0\le i\le m,\ 0\le j\le n} a_{ij}x^{i}y^{j}=0.$$
If one of the following conditions does not hold, then we say that {\it{Runge's condition is satisfied}} and we can apply the Runge's finiteness  result \cite{Runge}.\\
$({\bf i}).$ $a_{in}=a_{mj}=0,$ for all non zero indexes $i,j$\\
$({\bf ii}).$ $a_{ij}=0,$ for all $i,j$ such that $in+jm\ge mn.$\\
$({\bf iii}).$ The leading term  
$$\sum_{in+mj=mn} a_{ij}x^{i}y^{j}$$ is constant power of an irreducible polynomial in ${\mathbb{Z}}[x,y].$\\
$({\bf iv}).$ The algebraic function $y = y(x)$ defined by the equation $F (x, y) = 0,$ has only one class of conjugate Puiseux expansions.\\\\
The curves we study in the present paper do not satisfy condition $({\bf iii}),$
so they satisfy Runge's condition. Indeed, the leading terms are one of the following form: $y^2-x^4,$ $y^2-c^2x^4$ and $c(y^2-x^4),$ which are reducible polynomials in ${\mathbb{Z}}[x,y].$


\begin{thebibliography}{0}

\bibitem{Ayad} M.~Ayad, Sur le th${\acute{\textrm{e}}}$or${\grave{\textrm{e}}}$me de Runge, Acta Arith. {\bf 58} (1991), p. 203--209, \url{https://eudml.org/doc/206348}


\bibitem{Magma} Wieb~Bosma, John~Cannon, and Catherine~Playoust, The {M}agma algebra system. {I}. {T}he user language, J. Symbolic Comput., {\bf 24}, p. 235--265, 1997.

\bibitem{Chev}  C.~Chevalley, Un th\'{e}or\`{e}me d'arithm\'{e}tique sur les courbes alg\'{e}briques. C. R. Acad. Sci. Paris (1932).

\bibitem{draz} K.~A.~Draziotis, Practical solution of the Diophantine Equation $X^{nr}+Y^{n}=q,$ Elemente der Mathematik, {\bf 66}, p. 19--25 (2011), \url{https://ems.press/journals/em/articles/4355}.

\bibitem{draz2} K.~A.~Draziotis and D.~Poulakis, Explicit Chevalley-Weil theorem for affine plane curves, The Rocky Mountain Journal of Mathematics, 2009,  p. 49-70

\bibitem{draz3} K.~A.~Draziotis and D.~Poulakis, Solving the Diophantine equation $y^2= x(x^2-n^2)$, Journal of Number Theory 129 (1), 2009, p. 102-121.

\bibitem{duq}  S.~Duquesne, Points rationnels et méthode de Chabauty elliptique. (French) [Rational points and the elliptic Chabauty method] Les XXIIèmes Journées Arithmetiques (Lille, 2001). J. Théor. Nombres Bordeaux 15 (2003), no. 1, p. 99--113.

\bibitem{GH} {A.~Grytczuk and A.~Schinzel,} On Runge's Theorem about Diophantine equations,
in Colloq. Math. Soc. Janos Bolyai 60, North-Holland, 1991, p. 329--356.


\bibitem{HiSt}  D.~L.~Hilliker; E.~G.~Straus, Determination of
bounds for the solutions to those binary Diophantine equations that satisfy the hypotheses of Runge's theorem. Trans. Amer. Math. Soc. {\bf 280}  no.{\bf 2} (1983), p.637--657.


\bibitem{Le}  Mao Hua Le, A note on the integer solutions of hyperelliptic equations. Colloq. Math. 68 (1995), no. {\bf 2}, p. 171--177.

\bibitem{Masser} D.~W.~Masser, Polynomial Bounds for Diophantine Equations, Amer. Math. Monthly {\bf 93} (1986), p. 486--488.

\bibitem{osipov} N.~N.~Osipov and S.~D.~Dalinkevich, An Algorithm for Solving a Quartic Diophantine Equation Satisfying Runge’s Condition. Computer Algebra in Scientific Computing. CASC 2019. Lecture Notes in Computer Science, vol {\bf 11661}. Springer, Cham. \url{https://doi.org/10.1007/978-3-030-26831-2\_25}

\bibitem{osipov_russian} N.~N.~Osipov and M.~I.~Medvedeva, An Elementary algorithm for solving a diophantine equation
of degree four with Runge’s condition, J. Sib. Fed. Univ. Math. Phys. 2019. Vol. 12(3), p. 331--341.

\bibitem{poorten} Alf van der Poorten and Gerhard Woeginger, Squares from Products of Consecutive Integers, The American Mathematical Monthly
Vol. 109, No. {\bf 5} (May, 2002), p. 459--462.


\bibitem{poulakis} D.~Poulakis, A simple method for solving the Diophantine equation $y^2=x^4+ax^3+bx^2+cx+d$, Elem. Math. 54 (1999), p. 32 -- 36, \url{https://ems.press/content/serial-article-files/678}

\bibitem{Runge} C.~Runge, Über ganzzahlige Lösungen von Gleichungen zwischen zwei Veränderlichen. J. Reine Angew. Math., 100, p. 425--435, 1887.

\bibitem{sage}  SageMath, the Sage Mathematics Software System (Version 8.1). The Sage Developers, 2018, https://www.sagemath.org.

\bibitem{sprind} V.~G.~Sprindzuk, Classical Diophantine Equations. New York: Springer-Verlag, 1993.

\bibitem{Sch}  Wolfgang~M.~Schmidt, Eisenstein's theorem on power series expansions of algebraic functions. Acta Arith. {\bf 56} (1990), no. {\bf 2}, p.161--179.

\bibitem{laszlo} L.~Szalay, Fast Algorithm for solving superelliptic equations of certain types. Acta Acad. Paed. Agriensis, Sectio Mathematicae {\bf 27} (2000) p. 19--24. \url{http://real.mtak.hu/142527/1/AAPASM_27_from19to24.pdf}

\bibitem{laszlo2} L.~Szalay, Superelliptic equations of the form $y^p = x^{kp} + a_{ kp-1} x^{kp-1} + \cdots + a_0,$ Bull. Greek Math. Soc., 46, p. 23--33, 2002.

\bibitem{Tengely} S.~Tengely, Effective Methods for Diophantine Equations, Phd Thesis, Thomas Stieltjes Institute for Mathematics, 2005, \url{https://www.math.leidenuniv.nl/~tijdeman/tengely.pdf}.

\bibitem{tzanakis} Nikos Tzanakis, Elliptic Diophantine Equations, De Gruter 2013.

\bibitem{Walsh}  P.~G.~Walsh, A quantitative version of Runge's theorem on Diophantine equations. Acta Arith. {\bf 62} (1992), no. 2, p. 157--172.


\end{thebibliography}
\end{document}